\documentclass[USenglish,11pt]{article}	

\usepackage[utf8]{inputenc}				
\usepackage[T1]{fontenc}
\usepackage{microtype}
\usepackage[numbers,square,sort&compress]{natbib}
 
\usepackage{hyperref, amssymb, amsmath, amsthm, color}
\usepackage{geometry}
\newtheorem{thm}{Theorem}
\newtheorem{cor}{Corollary}
\newtheorem{prop}{Proposition}
\newtheorem{lem}{Lemma}
\newtheorem{asmpt}{Assumption}

\theoremstyle{definition}
\newtheorem{defi}{Definition}

\newtheorem{rmk}{Remark}

\newcommand{\KN}{\mathbin{\bigcirc\mspace{-15mu}\wedge\mspace{3mu}}}
\newcommand{\II}{I\! I}
\newcommand{\III}{I\! I\! I}
\newcommand{\cM}{\mathcal{M}}
\DeclareMathOperator{\Id}{Id}
\DeclareMathOperator{\tr}{tr}
\DeclareMathOperator{\proj}{proj}
\DeclareMathOperator{\Ric}{Ric}

\begin{document}


\title{Curvature of the Gauss map for normally flat submanifolds in space forms}

\author{Javier Álvarez-Vizoso}
\date{\today}

	


\maketitle
	
\abstract{For a submanifold with flat normal bundle in a space form  there is a normal orthonormal basis that simultaneously diagonalizes the corresponding Weingarten operators, and at which these operators satisfy a simple Codazzi symmetry. When the second fundamental form has zero index of relative nullity, the image of the Gauss map as defined by Obata is a submanifold of a generalized Grassmannian manifold, with induced metric given by the third fundamental form, corresponding to the sum of the squares of these operators. In this case, we show that the Riemann curvature tensor of the Gauss image is completely determined by the curvature and Weingarten operators of the original submanifold, in a form analogous to a \emph{theorema egregium} that relates the intrinsic geometry of the Gauss map with the extrinsic geometry of the original embedding. For this, we derive the difference vector of the induced Levi-Civita connections and express the Riemann curvature tensor with a generalized Kulkarni-Nomizu product.}


\section{Introduction}

Classically, the first fundamental form of a submanifold of Euclidean space is defined to be the restriction of the Euclidean scalar product to the submanifold tangent space. The second fundamental form of a hypersurface in turn is traditionally defined as the symmetric bilinear form that measures the tangent variation of the normal vector to the hypersurface when moving along tangent vectors. The corresponding Weingarten operator, or shape operator, is the linear map associated to this bilinear form with respect to the metric. The bilinear form defined by the square of this operator is the classical third fundamental form.

Considering the second fundamental form of a submanifold as a (pseudo-)Riemannian metric, there has been a series of studies, \cite{Aledo2007}\cite{Aledo2003}\cite{Aledo2005}\cite{haesen2007}\cite{haesen2010}\cite{verpoortTh}, that establish relationships between the curvature of the first and the second fundamental form metrics on a (hyper)surface, and compute explicit formulas for the scalar curvature in terms of extrinsically defined tensors via the Weingarten operator. In \citep{Aledo2005} the classical Liebmann-Hilbert rigidity theorem is proved in these terms, and \cite{haesen2010}\cite{verpoortTh} relate the scalar curvature to the mean curvature of the second fundamental form via deformations, in order to characterize, among other results, extrinsic hyperspheres and ovaloids. On the other hand, \citep{Liu1999} obtains a \emph{theorema egregium}-type formula relating the curvatures of the first and the third fundamental forms metrics for hypersurfaces in nonflat space-forms using relationships from the method of conjugate connections \cite{nomizu1991}.

Regarding the generalized third fundamental form to arbitrary submanifolds, classically studied for surfaces, e.g. \cite{Hartman1953}, Obata, \cite{obata1968gauss}, has extended the Gauss map to $n$-dimensional submanifolds isometrically immersed in space forms, $x:\mathcal{M}^n\rightarrow \overline{M}(k)^{n+m}$, using a general Grassmannian manifold $Q^{n+m} = G(n+m)/G(n)\times O(m)$, where $G$ is a suitable group, by associating to every point $p\in\cM$ the totally geodesic $n$-subspace tangent to $x(p)$ in $x(\cM)$. Given sufficient regularity conditions, e.g. if the Chern-Kuiper extrinsic zero index is null, cf. \cite{chernkuiper1952} and Sec. \ref{subsec:submanifolds}, this map $f:\cM^n\rightarrow Q^{n+m}$ yields a regular $n$-dimensional submanifold. Moreover, \cite{obata1968gauss} proves that the induced metric on this submanifold is the third fundamental form of $x(\cM)$. From this, computing the Riemann-Christoffel tensor of the third fundamental form metric corresponds to computing the curvature of the Gauss image. Similarly, \cite{nishikawa1975} generalizes the Gauss map and third fundamental form to Kähler immersions. Along these lines, \cite{muto1978} obtains necessary and sufficient conditions for the sectional curvature of the Grassmann manifold to vanish along a plane tangent to the Gauss image, for submanifolds of Euclidean space. Other works, e.g., \cite{Hoffman1982}\cite{vlachos2000}\cite{vlachos2007} further study (hyper)surfaces endowed with the third fundamental form metric. In particular, \cite{Vlachos2002} studies the curvature of the third fundamental form metric for minimal submanifolds of Euclidean space.

The purpose of the present work is to obtain a \emph{theorema-egregium}-type formula relating the curvature tensors of the Gauss image and the original submanifold in terms only of the shape operators of the original embedding, i.e. to relate the intrinsic geometry of $f(\cM)$ to the extrinsic geometry of $x(\cM)$.

Since the first and third fundamental forms do not seem to be conjugate connections in the sense of \cite{Liu1999}\cite{nomizu1991}, in codimension higher than one, we follow the coordinate-free method of \cite{Vlachos2002} determining the difference of the induced Levi-Civita connections in terms of Weingarten operators. Using this method on normally flat submanifolds, we arrive at a simplified version of \cite[Prop. 3.1]{Vlachos2002} that can be expressed in terms of the principal curvatures and principal normal vectors. Then, we obtain a closed formula for the full curvature tensor and scalar curvature, a result that complements \cite[Prop. 3.3]{Vlachos2002} for normally flat submanifolds in space forms, and that is amenable to be expressed in terms of a generalized Kulkarni-Nomizu product, resembling the structure of the classical Gauss equation.
The flat normal bundle requirement is assumed so that a simple Codazzi equation symmetry of the Weingarten operators covariant derivatives holds in higher codimension and, most importantly, so that there is a normal orthonormal basis at which the corresponding Weingarten operators are simultaneously diagonalizable.

The structure of the paper is as follows: in Section \ref{subsec:submanifolds} we introduce basic results on the geometry of submanifolds with flat normal bundle; in Section \ref{subsec:Obata} we define the third fundamental form in arbitrary codimension and state its interpretation as the induced metric on the Gauss image due to Obata. This sets the stage for the statement of our main theorem in Section \ref{subsec:statements}, relating the geometry of the submanifold to that of its Gauss map, to be proved in the rest of the paper. In Section \ref{subsec:Tvector} we explicitly compute the difference of the Levi-Civita connections induced from the first and third fundamental forms in terms of the Weingarten operators, and in Section \ref{subsec:principalNormals} we express this dependency in terms of the principal curvatures and principal normal vectors. In Section \ref{subsec:curvatureT} we prove our \emph{theorema-egregium}-type formula yielding the difference between the curvature tensors of the submanifold and its Gauss image, in terms of tensors determined solely from the shape operators; in Section \ref{subsec:curvatureKN} we introduce a generalized Kulkarni-Nomizu product that shows how some of these tensors factor and clarifies the similar role of our formula to the classical Gauss equation; in Section \ref{subsec:curvatureScalar} a formula for the scalar curvature of the Gauss map is obtained solely in terms of extrinsic scalars. 

\section{Preliminaries and statement of the main theorem}
\label{sec:preliminaries}

All throughout the text we shall make reference to results and equations of this first section.


\subsection{Submanifolds with flat normal bundle}
\label{subsec:submanifolds}

We shall assume standard notation on covariant derivatives, curvature tensors and the fundamental equations by Gauss, Codazzi-Mainardi and Ricci for submanifolds isometrically embedded in space forms. References for these results can be found, e.g., in \cite{kobayashiI,kobayashiII} and \cite{dajczer2019}.
We consdier our space forms to be the simply connected $(n+m)$-dimensional manifolds $\overline{M}(k)^{n+m}$ of constant sectional curvature such that $\overline{M}$ is: the Euclidean space $\mathbb{R}^{n+m}$ for $k=0$, the sphere $\mathbb{S}^{n+m}$ of radius $1/\sqrt{k}$ for $k>0$, and the hyperbolic space $\mathbb{H}^{n+m}$ of curvature $k<0$.

Let us review some of the essential properties of submanifolds which have flat normal bundle.

%
%
\begin{defi}
An isometric immersion $x:\mathcal{M}^n\rightarrow \overline{M}(k)^{n+m}$, of an $n$-dimensional manifold $\cM$ into an $(n+m)$-dimensional space form of constant sectional curvature $k$, is said to be normally flat, or to have flat normal connection, if the curvature tensor of the normal connection, $\nabla^\perp_X \xi = (\overline{\nabla}_X\xi)^\perp$, is zero for every $X,Y\in\mathfrak{X}(\cM),\, \xi\in\Gamma(N\cM)$, i.e., $$R^\perp (X,Y)\xi = \nabla_X^\perp\nabla_Y^\perp \xi - \nabla_Y^\perp\nabla_X^\perp \xi -\nabla^\perp_{[X,Y]}\xi = 0.$$
\end{defi}

\begin{prop}
The normal connection $\nabla^\perp$ of $x:\mathcal{M}^n\rightarrow \overline{M}(k)^{n+m}$ is flat if and only if at every point $p\in\cM$ there is an orthonormal basis $\{\xi_j\}_{j=1}^m$ of $N_p\cM$ that is parallel in the normal bundle, i.e., $$\nabla_X^\perp\xi_j = 0,\quad\forall X\in T_p\cM,\quad j=1,\dots,m.$$
\end{prop}
\begin{proof}
See \cite[Ch. 4, Prop. 1.1]{Chen2019}.
\end{proof}
\begin{prop}
There is a normal orthonormal basis $\{ \xi_j\}_{j=1}^m$ of $T_p\cM$, such that the Weingarten operators $\{ A_{\xi_j}\}_{j=1}^m$ are simultaneously diagonalizable, if and only if the isometric immersion $x:\mathcal{M}^n\rightarrow \overline{M}(k)^{n+m}$ has flat normal bundle.
\end{prop}
\begin{proof}
This immediately follows from the Ricci equation.
\end{proof}
From the simultaneous diagonalization of the shape operators, the tangent space decomposes orthogonally at every $p\in\cM$ as $T_p\cM=E_1(p)\oplus\cdots\oplus E_s(p)$, for some $1\leq s(p)\leq n$, such that for every normal vector $\xi\in N_p\cM$ there are real numbers $\lambda_i(\xi), 1\leq i\leq s$, called \emph{principal curvatures}, such that
\begin{equation}\label{eq:principalCurvatures}
A_\xi\vert_{E_i(p)} = \lambda_i(\xi)\Id
\end{equation}
and the maps $\xi\mapsto\lambda_i(\xi)$ are pairwise distinct. These maps are linear so there exist unique pairwise distinct normal vectors $\eta_i(x)\in N_x\cM$, the \emph{principal normal} vectors, satisfying
\begin{equation}\label{eq:principalNormals}
\lambda_i(\xi) = \langle\, \eta_i(p), \xi \,\rangle, \quad 1\leq i \leq s.
\end{equation}
This makes the Weingarten operators take the simple expression
\begin{equation}\label{eq:Aprincipals}
 A_\xi X = \sum_{i=1}^s \langle\, \xi, \eta_i(p) \,\rangle\proj_i X,\quad \forall X\in T_p\cM,\; \forall\xi\in N_p\cM.
\end{equation}

\begin{rmk}\label{rmk:productRule}
The canonical connection on $Hom^2(TM, NM; TM)$ yields a covariant derivative for the shape operators
$$
(\nabla_X A_\xi)X = \nabla_X (A_\xi X) - A_\xi\nabla_X Y - A_{\nabla^\perp_X\xi}Y,
$$
that for normally flat submanifolds, and $\xi$ one of its parallel normal vectors, leads us to use the following result and notation for the rest of this paper
\begin{equation}\label{eq:Aderivative}
(\nabla_X A_\xi)X = \nabla_X (A_\xi X) - A_\xi\nabla_X Y,
\end{equation}
i.e., we can work with this expression, for expanding or collecting parentheses, as if a product rule is taking place.
\end{rmk}
Finally, all along the proofs of the paper we shall make extensive use of the Codazzi symmetries of this covariant derivative, which for our purpose are enough to be true when $\xi$ is a parallel normal vector.

\begin{prop}[Codazzi symmetries]
The Weingarten operators of a normally flat submanifold in a space form satisfy:
\begin{equation}\label{eq:Codazzi1}
\langle\, (\nabla_X A_\xi) Y,\, Z \rangle = \langle\, (\nabla_Y A_\xi) X,\, Z \rangle = \langle\, (\nabla_X A_\xi) Z,\, Y \rangle,
\end{equation}
for all $ X,Y, Z\in \mathfrak{X}(M)$ and $\xi$ a parallel vector in the normal connection.
\end{prop}
\begin{proof}
The first equality is the traditional Codazzi symmetry satisfied by submanifolds of a manifold of constant sectional curvature, see \cite[Eq. 1.16]{dajczer2019}. The second equality is the self-adjointness of the shape operator covariant derivative, which follows straightforwardly from $A$ being self-adjoint (e.g., see also Lemma \ref{lem:Obata}) below.
\end{proof}
%

\subsection{The Gauss-Obata map and the third fundamental form}
\label{subsec:Obata}

\begin{defi}
Let $x:\cM^n\rightarrow \overline{\cM}^{n+m}$ be a submanfiold of a Riemannian manifold, and let $\{A_{\xi_j}\}_{j=1}^m$ be the Weingarten operators at an orthonormal basis $\{\xi_j\}_{j=1}^m$ of the normal bundle $N\cM$. The \emph{Obata operator} is defined to be
\begin{equation}
W=\sum_{j=1}^m A^2_{\xi_j} : T\cM \rightarrow T\cM.
\end{equation}
\end{defi}

\begin{lem}\label{lem:Obata}
The Obata operator satisfies the following properties for any $ X,Y,Z\in \mathfrak{X}(M)$:
\begin{enumerate}
\item It is self-adjoint with respect to the first fundamental form:
\begin{equation}
\langle\, W X, Y\, \rangle = \langle\, X, W Y\, \rangle.
\end{equation}
\item Its covariant derivative is self-adjoint:
\begin{equation}\label{eq:ObataSym}
\langle\, (\nabla_X W) Y, Z\, \rangle = \langle\, (\nabla_X W) Z,\, Y \rangle.
\end{equation}
\item Its component with respect to the first fundamental form are independent of the normal orthornomal basis chosen for the shape operators, i.e.:
\begin{equation}
\langle\, W X, Y\, \rangle = \sum_{j=1}^m \langle\, A_{\xi_j}X ,\, A_{\xi_j}Y \,\rangle = \sum_{j=1}^m \langle\, A_{n_j}X ,\, A_{n_j}Y \,\rangle
\end{equation}
for any other orthonormal basis $\{n_j\}_{j=1}^m$ of $N\cM$.
\end{enumerate}
\end{lem}
\begin{proof}
The properties follow straightforwardly by direct computation. For example (ii) follows from $X\III(Y,Z) = X\III(Z,Y)$, since then
\begin{align*}
 \langle \nabla_X(WY),Z \rangle  + \langle WY, \nabla_X Z \rangle = \langle \nabla_X(WZ),Y \rangle  + \langle WZ, \nabla_X Y \rangle
\end{align*}
which by expanding the derivative of the product gives
\begin{align*}
 \langle (\nabla_XW)Y,Z \rangle + \langle W\nabla_X Y, Z\rangle + \langle WY, \nabla_X Z \rangle = \langle (\nabla_XW)Z,Y \rangle + \langle W\nabla_X Z, Y\rangle + \langle WZ, \nabla_X Y \rangle
\end{align*}
and most of the terms cancel by self-adjointess.
\end{proof}
In view of properties $(i)$ and $(iii)$ the following is well-defined.

\begin{defi}
The \emph{third fundamental form} $\III$ of a Riemannian submanifold $x:\cM^n\rightarrow \overline{\cM}^{n+m}$ is the symmetric bilinear form associated to the Obata operator with respect to the induced metric:
\begin{equation}
\III(X,Y) = \langle\, W X, Y\, \rangle,
\end{equation}
for all $ X,Y\in \mathfrak{X}(M)$.
\end{defi}
%

Now, let us review the definitions and results from \cite{obata1968gauss} about the generalized Gauss map. Consider the group $G(n+m)$ associated to the corresponding space form as one of the following: the orthogonal group $O(n+m+1)$ for $k>0$, the group $\mathbb{E}(n+m)$ of Euclidean motions of $\mathbb{R}^{n+m}$ for $k=0$, the group $O(1,n+m)$ of inhomogeneous Lorentz transformations on $\mathbb{H}^{n+m}$ for $k<0$. Let $Q$ be the set of all totally geodesic $n$-spaces in $\overline{M}$, then $G(n+m)$ acts on $\overline{M}$ transitively. In fact, $Q$ can be identified with a homogeneous space
\begin{equation}
Q^{n+m} = G(n+m)/G(n)\times O(m).
\end{equation}
In the case $G(n+m)= O(n+m+1)$, Q is the classical Grassmann manifold and the association of every point $p\in\cM\subset\overline{M}$ to the totally geodesic $n$-spaces correspond to the classical Gauss map. We consider the assignment in the other cases as a \emph{generalized Gauss(-Obata) map}, i.e., to $x:\mathcal{M}^n\rightarrow \overline{M}(k)^{n+m}$ we associate the map $f:\mathcal{M}^n\rightarrow (Q^{n+m},\sigma)$, where $f(p)$ is the totally geodesic $n$-space tangent to $x(\cM)$ at $x(p)$, and $\sigma$ is the standard (pseudo)-Riemannian metric of $Q$ as explained in \cite{obata1968gauss}.

The image of $f$ is called the \emph{Gauss image} in the generalized Grassmannian. The \emph{index of relative nullity}, or Chern-Kuiper extrinsic zero index, cf. \cite{chernkuiper1952}, at a point $p\in\cM$, is the dimension of the maximal subspace in $T_p\cM$ on which the second fundamental form $\II$ of $x(\cM)$ is zero. It can readily be seen that if $\nu=0$ on the whole submanifold, then the Gauss image will be a regular $n$-dimensional submanifold in $Q$ and the rank of $f$ will be maximal.
\begin{asmpt}
In the rest of the present work we consider always submanifolds with $\nu=0$ or restrict to the domain where this is valid. Moreover, we assume the restriction to the subdomain of $x(\cM)$ such that sufficient regularity for the Obata operator $W$ holds, in order to induce a non-degenerate bilinear form $\III$, and to have a well defined inverse $W^{-1}$ at every point of consideration.
\end{asmpt}

The motivation for the previous definitions and the rest of the paper is the following 

\begin{thm}[Obata \cite{obata1968gauss}]
The induced metric $f^*\sigma$ on the Gauss image of $f:\mathcal{M}^n\rightarrow (Q^{n+m},\sigma)$, of an isometrically immersed manifold in a space form, $x:\mathcal{M}^n\rightarrow \overline{M}(k)^{n+m}$, is equal to the third fundamental form of $x(\cM)$, and satisfies the following relation with respect to the Ricci tensor of $\cM$, and the second fundamental form of $x(\cM)$ in the direction of the mean curvature vector $H$:
\begin{equation}\label{eq:Obata}
(f^*\sigma)(X,Y) = \III(X,Y) = k(n-1)\langle\, X,\, Y\,\rangle + \langle\, \II(X,Y),\, H \,\rangle - \Ric(X,Y)
\end{equation}
for any $ X,Y\in \mathfrak{X}(M)$.
\end{thm}
We denote by $\,\widetilde{\;}\,$ the operator associated to a bilinear form with respect to the metric, i.e. $$W = \widetilde{\III},\; A_\xi = \widetilde{\langle\II,\xi\rangle}, \text{ and }\Ric(X,Y) = \langle\,\widetilde{\Ric}\,X,Y\,\rangle.$$ 
\begin{cor}
The operator $W$ is related to the Ricci operator $\widetilde{\Ric}$ of $\cM$ and the Weingarten operator at the mean curvature vector, $A_H$, by
\begin{equation}\label{eq:ObataRicci}
W = k(n-1)\Id + A_H - \widetilde{\Ric}.
\end{equation}
\end{cor}

\begin{rmk}
Obata obtains $f^*\sigma$ and defines $\III$ in terms of the second fundamental form using the language of the Maurer-Cartan forms of $Q$. The last corollary justifies the definition and naming of the $W$ operator and the third fundamental form as presented here, which naturally generalize the classical definition for hypersurfaces as the square of the Weingarten operator, cf. \cite[Def. 4.1]{Liu1999}. Moreover \ref{eq:ObataRicci}, already appearing in terms of shape operators in e.g. \cite[eq. 1.4]{Matsuyama1988}, is the generalization to arbitrary codimension of the classical well-known formula  $K\Id - H\II + \III = 0$, ($K$ is the Gaussian curvature), for surfaces in $\mathbb{R}^3$, cf. \cite[eq. 31]{Hartman1953} or \cite[Th. 2.5.1]{toponogov2006}, and for hypersurfaces in $\mathbb{R}^{n+1}$ \cite[Prop. 5.2]{kobayashiII}, $A^2=HA -\widetilde{\Ric}$.
\end{rmk}


\subsection{Overview of the main results}
\label{subsec:statements}

We collect here the statements of our main results to be proved in the rest of the paper (see respectively Sections \ref{subsec:Tvector}, \ref{subsec:curvatureT} and \ref{subsec:curvatureScalar}).
\begin{thm}\label{th:main}
For a submanifold with flat normal bundle in a space form, $x:\mathcal{M}^n\rightarrow \overline{M}(k)^{n+m}$, such that its Gauss-Obata image $f:\mathcal{M}^n\rightarrow Q^{n+m}$ is a smooth submanifold, let $A_j=A_{\xi_j}$, be the Weingarten operators at the parallel orthonormal normal basis vectors $\{\xi_j\}_{j=1}^m$, and $W=\sum_{j=1}^m A^2_j$ the Obata operator of the third fundamental form metric, $\III$, in $f(\cM)$ induced from $Q$. Then the following comparison tensors between the intrinsic induced geometries of $x$ and $f$ are completely determined by the extrinsic geometry of $x$:
\begin{enumerate}
\item The difference of the Levi-Civita connections in $\cM$ of the induced metrics from $Q$ and $\overline{M}(k)$ is
\begin{equation}\label{eq:comparisonT}
\nabla^{(\III)}_XY - \nabla_XY = W^{-1}\sum_{j=1}^m A_j(\nabla_X A_j)Y, \quad \text{for any } X,Y\in \mathfrak{X}(M).
\end{equation}
\item\emph{(Theorema Egregium)} The difference between the Riemann-Christoffel curvature tensor of the Gauss-Obata image $(f(\mathcal{M}),\III)$ and the curvature tensor of the submanifold $(x(\mathcal{M}),I)$ is given by
\begin{align}\label{eq:comparisonR}
\III (R^{(\III)}(X,Y)Z - & R(X,Y)Z, V) \nonumber\\ & = \III( P(X,Y)Z, V) + \III ( T(X,Z),\, T(Y,V) ) -  \III( T(X,V),\, T(Y,Z)),
\end{align}
for all $X,Y,Z,V\in \mathfrak{X}(M)$, where
$
P(X,Y)Z = W^{-1}\sum_{j=1}^m\left( A_j R(X,Y)A_j  + \left[ \nabla_X A_j,\,\nabla_Y A_j \right]\, \right)Z,
$
and $T(X,Y)$ is the right hand side of (\ref{eq:comparisonT}).

\item The scalar curvature of $(f(\mathcal{M}),\III)$ satisfies:
\begin{equation}\label{eq:scalarR}
\mathcal{R}^{(\III)} = k(n-1)\frac{\sigma_{n-1}(W)}{\det(W)} + \tr_{\III}(A_{H}) - n + ||T||^2_{\III} -||\tr_{\III}T||^2_{\III} + \tr_{\III}\tr_c P.
\end{equation}
\end{enumerate}
\end{thm}
Here $H=\tr_I I\!I$ is the mean curvature vector of $x(\mathcal{M})$, $\sigma_k$ is the $k$-th elementary symmetric polynomial on the eigenvalues of an operator, $||\cdot||_g$ is the induced norm by a metric $g$, $\tr_g$ is the trace of an operator or bilinear form with respect to the inner product $g$, and $\tr_c P=\tr(V\mapsto P(V,X)Y)$ is the Ricci contraction.
In fact, the dependencies on the $(\nabla_X A_j)Y$ of any of these formulas can in principle be written in terms only of the principal curvatures and principal normal vectors of $x(\cM)$, and $\nabla_XY$. We show explicit formulas for $T(X,Y)$ in Section \ref{subsec:principalNormals}.


There is an alternative expression of equation \ref{eq:comparisonR} with the curvature decomposed in terms of algebraic curvature tensors coming from a generalized Kulkarni-Nomizu product: by splitting up the tensor $P$ into a part corresponding to the curvature of the Weingarten operators, and a part corresponding to the commutator term rewritten using vector-valued bilinear forms (see Sec. \ref{subsec:curvatureKN}), each respectively given by:
$$
L(X,Y)=W^{-1}\sum_{j=1}^m A_j R(X,Y)A_j,\quad\, J_j(X,Y)=W^{-\frac{1}{2}}(\nabla_X A_j)Y,\quad\,\text{ for all }X,Y\in \mathfrak{X}(M).$$
\begin{cor}\label{cor:main}
The curvature tensor of the Gauss map of a normally flat submanifold in a space form with null extrinsic zero index, admits a decomposition in terms of the original curvature, the curvature of the shape operators, and generalized Kulkarni-Nomizu products determined by these operators:
\begin{equation}\label{eq:Roperators}
\boxed{
R^{(\III)} = R + L + \widehat{T\KN_{\III} T} - \sum_{j=1}^m \widehat{J_j\KN_{\III} J_j}
}
\end{equation}
\end{cor}
Here $\widehat{S}$ is the $(1,3)$-tensor corresponding to a $(0,4)$-tensor $S$ with respect to the metric $\III$ via $$\III( \widehat{S}(X,Y)Z, V )=S(X,Y,Z,V),\quad\text{ for all } X,Y,Z,V\in \mathfrak{X}(M),$$
and $\KN_{g}$ is the generalized Kulkarni-Nomizu product, see Definition \ref{def:KN}, between bilinear forms valued in an inner product space $(E,g)$. In our case $T$ and $J_j$ can be identified with elements of $S^2(T^*_p\mathcal{M})\otimes T_p\mathcal{M}$, where the inner product space is $(T_p\mathcal{M},\III)$.
\begin{rmk}
Formulas \ref{eq:comparisonR} and \ref{eq:Roperators} for the curvature of the Gauss image in terms of the submanifold shape operators play the role of a \emph{theorema egregium}, analogous to the Gauss equation for the curvature of the original submanifold in terms of the second fundamental form. Indeed, both provide the difference of intrinsic curvature invariants, i.e. the Riemann-Christoffel tensors, in terms of a combination of extrinsic invariants of the submanifold: the Gauss equation in terms of combinations of $A_j$ components via products of $\II$, our results in terms of products of $A_j$ and $\nabla A_j$ within the tensors involved. This is perhaps most clear when writing the Gauss equation in terms of the generalized Kulkarni-Nomizu product:
\begin{equation}\label{eq:GaussKN}
\overline{R} = R + \widehat{\II\KN_I \II},
\end{equation}
so that the analogy to the decomposition of equation \ref{eq:Roperators} is manifest, considering that $\II$ plays a similar role with respect to the covariant derivatives as $T(X,Y) = \nabla^{(\III)}_XY -\nabla_XY$ does, since $\II(X,Y) = \overline{\nabla}_XY -\nabla_XY$. The other tensors $L$ and $J_j$ reflect the higher complexity of relating the geometry of two submanifolds with different metrics in different ambient manifolds, $(x(\mathcal{M}),I)$ in $\overline{M}(k)^{n+m}$ and $(f(\mathcal{M}),\III)$ in $G(n+m)/G(n)\times O(m)$, instead of one submanifold and its ambient space with the same metric.
\end{rmk}
\begin{rmk}
Notice the following interpretation of each term: the vector-valued bilinear form $T$ is an obstruction to the first and third fundamental forms inducing the same Levi-Civita connection. The tensor $L$ extracts all the dependencies of $R^{(\III)}$ on the second order covariant derivatives of the shape operators precisely only in terms of the curvature of these operators. If the original submanifold is locally flat, $R+L$ can be shown to be zero, so the algebraic curvature tensor in the $J_j$ represents an independent obstruction to the Gauss image being also locally flat.
\end{rmk}

\section{Comparison of the induced Levi-Civita connections}
\label{sec:connections}

In this section we determine the vector difference of the covariant derivatives induced by the first and third fundamental forms in order to prove part $(i)$ of the main Theorem \ref{th:main}. This can be obtained from simplifying \cite[Prop. 3.1]{Vlachos2002} in the case of submanifolds with flat normal bundle in Euclidean space. We include a full proof for completeness and to show how the normally flat condition simplifies the derivation. We also provide the explicit expression of this vector in terms of the principal curvatures and principal normal vectors.


\subsection{Connection difference in terms of the Weingarten operators}
\label{subsec:Tvector}
\begin{defi}
Let us denote by $T$ the vector-valued symmetric bilinear form $T\in S^2(T^*\cM)\otimes T\cM$ measuring the difference between the Levi-Civita connections induced by the third and first fundamental forms:
\begin{equation}
T(X,Y) = \nabla^{(\III)}_XY -\nabla_XY.
\end{equation}
\end{defi}

{\bf Proof of Theorem \ref{th:main} (i)}
Let us recall that both $\nabla$ and $\nabla^{(\III)}$ are torsion-free so that
\begin{equation}\label{eq:brackets}
[X,Y] = \nabla^{(\III)}_XY -  \nabla^{(\III)}_YX = \nabla_XY -  \nabla_YX ,
\end{equation}
which implies that $T$ is symmetric, $T(X,Y)=T(Y,X)$. Moreover, the connections are metric compatible, thus 
\begin{align}
& X(\III(Y,Z)) = \III(\nabla^{(\III)}_XY,Z) +  \III(\nabla^{(\III)}_XY,Z), \text{ and} \\ 
& X\langle Y,Z\rangle = \langle\nabla_XY,Z\rangle + \langle\nabla_XY,Z\rangle.\label{eq:compatible}
\end{align}
Now, Koszul's formula can be used to define the unique connection satisfying the previous properties, hence
\begin{align*}
2\III(\nabla^{(\III)}_XY,Z) & = X\III(Y,Z) + Y\III(X,Z) - Z\III(X,Y) - \III([Y,X],Z) - \III([X,Z],Y)- \III([Y,Z],X),
\end{align*}
which expanding the metric in terms of the Obata operator (for ease of notation, let us write $A_j=A_{\xi_j}$)
\begin{align*}
2\sum_{j=1}^m \langle A_j^2(\nabla^{(\III)}_XY),Z\rangle & = \sum_{j=1}^m\left( X\langle A^2_jY,Z \rangle + Y\langle A^2_jX,Z \rangle - Z\langle A^2_j X,Y \rangle  \right. \\
& \quad\quad\quad\quad \left. - \langle A^2_j[Y,X],Z \rangle - \langle A^2_j[X,Z],Y \rangle - \langle A^2_j[Y,Z],X \rangle\right) .
\end{align*}
By expanding the the first three terms of the right hand side with eq. \ref{eq:compatible}, the Lie brackets with eq. \ref{eq:brackets}, and making use of the self-adjoint property of the $A_j$, one can cancel and collect terms to arrive at:
\begin{align*}
2\sum_{j=1}^m \langle A_j^2(\nabla^{(\III)}_XY),Z\rangle & = \sum_{j=1}^m\left( \langle\nabla_X(A_j^2Y),Z\rangle +\langle\nabla_Y(A_j^2X),Z\rangle  - \langle\nabla_Z(A_j^2X),Y\rangle \right. \\
& \quad\quad\quad\quad \left. - \langle \nabla_YX-\nabla_XY, A^2_jZ \rangle + \langle \nabla_ZX,A_j^2Y \rangle \right),
\end{align*}
which in terms of the Obata operator is
\begin{align*}
2\III(\nabla^{(\III)}_XY,Z) & = \langle \nabla_X(WY),Z\rangle + \langle \nabla_Y(WX),Z\rangle - \langle \nabla_Z(WX),Y\rangle - \langle W(\nabla_YX-\nabla_XY), Z \rangle + \langle W\nabla_ZX, Y \rangle.
\end{align*}
Expanding the first three terms of the right hand side using $\langle \nabla_X(WY),Z\rangle = \langle W(\nabla_XY),Z\rangle + \langle(\nabla_XW)Y,Z\rangle$, and canceling terms, we obtain
\begin{align*}
2\III(\nabla^{(\III)}_XY,Z) & = 2\langle W(\nabla_XY),Z\rangle + \langle(\nabla_XW)Y,Z\rangle - \langle(\nabla_ZW)X,Y\rangle +\langle(\nabla_YW)X,Z\rangle,
\end{align*}
that is to say, by definition of $T$:
\begin{equation}\label{eq:tempT}
 2\langle W T(X,Y),Z \rangle = \;\langle(\nabla_XW)Y,Z\rangle - \langle(\nabla_ZW)X,Y\rangle +\langle(\nabla_YW)X,Z\rangle.
\end{equation}
But, employing equation \ref{eq:Codazzi1} and self-adjointness, the second term of the right hand side is
\begin{align*}
\langle \nabla_Z\left(\sum_{j=1}^m A_j^2\right)X, Y \rangle & = \sum_{j=1}^m \langle\;\left( (\nabla_Z A_j)A_j+A_j(\nabla_Z A_j) \right) X, Y \rangle \\
& = \sum_{j=1}^m \langle\; (\nabla_ZA_j)Y, A_jX\rangle  + \langle\; (\nabla_XA_j)Z, A_jY\rangle \\ 
& = \sum_{j=1}^m \langle\; (\nabla_YA_j)Z, A_jX\rangle  + \langle\; (\nabla_XA_j)(A_jY), Z\rangle \\ 
& = \sum_{j=1}^m \langle\; (\nabla_YA_j)(A_jX),Z \rangle  + \langle\; (\nabla_XA_j)(A_jY), Z\rangle.
\end{align*}
Then, since substituting in equation \ref{eq:tempT} yields an equation with all terms paired with $\langle\cdot,Z\rangle$, for all $Z\in\mathfrak{X}(\cM)$, we must have that:
\begin{align*}
2WT(X,Y) & = (\nabla_XW)Y + (\nabla_YW)X - \sum_{j=1}^m\left[\; (\nabla_YA_j)(A_jX) + (\nabla_XA_j)(A_jY)\;\right] \\
& = \sum_{j=1}^m\left[\; A_j(\nabla_XA_j)Y + A_j(\nabla_YA_j)X \;\right]
\end{align*}
from the covariant derivative of $W$ expressed as the derivative of a sum of squares and canceling one of the terms. Therefore, applying the Codazzi symmetry again and solving for $T$ by inverting $W$, since we are assuming regularity of $W$, equation \ref{eq:comparisonT} follows.
\qed

\begin{rmk}
Expression \ref{eq:comparisonT} is given explicitly in terms of the Weingarten operators at a parallel normal orthonormal basis $\xi$, that always exists for a normally flat submanifold \cite[cf. Ch. 4, Prop. 1.1 ]{Chen2019}, but the left hand side is independent of it. Indeed, it can be easily shown that for a new basis given by an orthonormal linear combination of the especial basis $\xi$, the right hand side is independent of the basis.
\end{rmk}


\subsection{Expression in terms of principal curvatures and principal normal vectors}
\label{subsec:principalNormals}

We shall express the covariant derivative of Weingarten operators at a normal vector $\xi$ in terms of the covariant derivative of the vectors involved, the principal curvatures $\lambda_i(\xi)$, and principal normal vectors $\eta_i$, $i=1,\dots, s$, following the notation of Sec. \ref{subsec:submanifolds}. The objective of this is to show that, although complicated in practice, the tensors involved in the \emph{theorema-egregium}-type formula \ref{eq:comparisonR} are in principle dependent only on the aforementioned quantities, i.e., the curvature of the Gauss image is given by the curvature of the original submanifold and an expression in terms only of the vectors covariant derivative and the principal curvatures and principal normal vectors of the original submanifold, which are purely extrinsic measures of the embedding's geometry.

\begin{lem}\label{lem:shapeprinc}
The covariant derivatives of the Weingarten operators $A_\xi$ of $x:\cM\rightarrow\overline{M}(k)$, at a parallel normal vector $\xi$,  expand in terms of the principal curvatures and principal normal vectors as:
\begin{equation}\label{eq:covprincipal1}
(\nabla_X A_\xi)Y = \sum_{i=1}^s\left( X(\lambda_i(\xi))\proj_i Y + \langle\xi,\eta_i\rangle\left[\nabla_X,\proj_i\right]Y \right),
\end{equation} 
where $\left[\nabla_X,\proj_i\right]=\nabla_X\circ\proj_i -\proj_i\circ\nabla_X$.
Focusing on the action on the components $X_i, Y_j, Z_k$, from the subspaces $E_i, E_j$ and $E_k$, this specializes to:
\begin{align}\label{eq:covprincipal2}
\langle (\nabla_{X_i} A_\xi)Y_j, Z_k\rangle = 
\left\{
\begin{array}{ll}
\langle \xi,\eta_j-\eta_k \rangle\langle \nabla_{X_i} Y_j, Z_k\rangle \quad\text{ if $j\neq k$}, \\
X_i(\lambda_j(\xi))\langle Y_j, Z_j\rangle\quad\quad\quad\;\, \text{ if $j = k$.}
\end{array}\right.
\end{align}
\end{lem}
\begin{proof}
From equations \ref{eq:Aprincipals} and \ref{eq:Aderivative} one obtains \ref{eq:covprincipal1}:
$$
\nabla_X(A_\xi Y) - A_\xi\nabla_X Y = \sum_{i=1}^s\left( X\langle\xi,\eta_i\rangle\proj_i Y+ \langle\xi,\eta_i\rangle\nabla_X(\proj_i Y) \right) - \sum_{i=1}^s \langle\xi,\eta_i\rangle\proj_i[\nabla_XY].
$$
The case with $j=k$ in \ref{eq:covprincipal2} is obtained by using \ref{eq:principalCurvatures} and \ref{eq:Aprincipals} while differentiating $\langle A_\xi Y_j, Z_j \rangle = \langle \xi, \eta_j \rangle\langle Y_j, Z_j\rangle$. The case $j\neq k$ is obtained from $\langle A_\xi Y_j, Z_k \rangle =0$, by orthogonality of the $E_j, E_k$ subspaces and differentiating, using self-adjointness:
\begin{align*}
\nabla_{X_i}\langle A_\xi Y_j, Z_k \rangle= \langle (\nabla_{X_i}A_\xi)Y_j,Z_k \rangle +\langle \xi, \eta_k \rangle\langle \nabla_{X_i}Y_j,Z_k\rangle +\langle \xi, \eta_j \rangle\langle Y_j, \nabla_{X_i}Z_k\rangle,
\end{align*}
notice that in the last term $\langle Y_j, \nabla_{X_i}Z_k\rangle=-\langle \nabla_{X_i}Y_j,Z_k\rangle$, by differentiating $\langle Y_j,Z_k\rangle=0$ by $X_i$.
\end{proof}

\begin{rmk}
The Codazzi symmetries of equation \ref{eq:Codazzi1} translate into symmetries of the principal curvatures and principal normals in equation \ref{eq:covprincipal2}.
\end{rmk}

\begin{prop}\label{prop:Tinprinc}
The vector-valued symmetric bilinear form $T$ can be expressed in terms of the principal curvature normal vectors of the Weingarten operators of $x:\cM\rightarrow\overline{M}(k)$ as:
\begin{equation}\label{eq:Texplicit1}
T(X,Y) = \sum_{i=1}^s\left( X(\log ||\eta_i||_g)\proj_i Y + \sum_{k=1}^s\frac{\langle\eta_i,\eta_k\rangle}{\langle\eta_k,\eta_k\rangle}\proj_k\circ\left[\nabla_X,\proj_i\right] Y\right).
\end{equation}
for any $X,Y\in \mathfrak{X}(\cM)$. Moreover, focusing on the action on the components $X_i, Y_j, Z_k$ from the subspaces $E_i, E_j$ and $E_k$, this specializes to:
\begin{align}\label{eq:Texplicit2}
\langle\,T(X_i,Y_j),Z_k\rangle = 
\left\{
\begin{array}{ll}\displaystyle
\frac{\langle \eta_j-\eta_k,\eta_k \rangle}{\langle\eta_k,\eta_k\rangle}\langle\,\nabla_{X_i}Y_j,Z_k\,\rangle,\quad\text{ if $j\neq k$} \\ \,\\
X_i(\log||\eta_j||)\langle Y_j,Z_j\rangle,\quad\quad\quad\;\text{ if $j= k$}
\end{array}
\right.
\end{align}
where $\{V_{k\mu}\}_{\mu=1}^{\dim E_k}$ are the $\III$-orthornomal basis of $E_k$ formed by the corresponding eigenvectors of $W$.
\end{prop}
\begin{proof}
The action of $W^{-1}$ on a vector $V_i\in E_i$ yields $w_i^{-1}V_i$, where $w_i=\sum_{j=1}^m\lambda^2_i(\xi_j)=\sum_{j=1}^m\langle\xi_j,\eta_i\rangle^2$. Applying the shape operator on \ref{eq:covprincipal1}, and decomposing the last term into the $\oplus_k E_k$ eigenspaces, leads to
\begin{align*}
T(X,Y) = \sum_{j=1}^m\sum_{i=1}^s\frac{1}{2} \frac{X(\lambda^2_i(\xi_j))}{w_i}\proj_i Y + \sum_{j=1}^m\sum_{i=1}^s\sum_{k=1}^s \frac{ 
\langle\xi_j,\eta_i\rangle\langle\xi_j,\eta_k\rangle }{\sum_{l=1}^m\langle\xi_l,\eta_k\rangle^2} \proj_k\circ [ \nabla_X,\proj_i ] Y
\end{align*}
which results in the stated equation by summing over the orthonormal basis $\xi_j$ and rewriting the $X$ derivative.

In order to get \ref{eq:Texplicit2} we apply $A_j$ and $W^{-1}$ over equation \ref{eq:covprincipal2}, taking into account that $W$ is self-adjoint and injective as an endomorphism, due to our regularity assumptions, so that $W^{-1}$ is self-adjoint as well, cf. \cite[App. A Prop. 8.2]{Taylor1}. If $j\neq k$ then the eigenvalues of $Z_k$ yield
\begin{align*}
\langle\,T(X_i,Y_j),Z_k\rangle = \sum_{l=1}^m\langle(\nabla_{X_i}A_l)Y_j,A_l|_{E_k}W|_{E_k}^{-1}Z_k \rangle = \sum_{l=1}^m\frac{\langle\xi_l,\eta_k\rangle}{w_k}\langle\xi_l,\eta_j-\eta_k\rangle\langle\nabla_{X_i}Y_j,Z_k \rangle,
\end{align*}
and the formula follows by summing over the $\xi_l$ basis. In the case $j=k$ the formula is obtained by the same reasoning.
\end{proof}

\section{Theorema Egregium for the Gauss-Obata map}
\label{sec:curvatures}

In \cite[Prop. 3.3]{Vlachos2002} a formula for $\III(R^{(\III)}(X,Y)Y,X)$ is obtained for submanifolds of Euclidean space. In this section we compute the full Riemann-Christoffel curvature tensor $\III(R^{(\III)}(X,Y)Z,V)$ for normally flat submanifolds in space forms, in order to compare the curvature of the two induced metrics on $\cM$, one coming from the isometric embedding in the space form and the other from the Gauss-Obata map into a generalized Grassmannian manifold. We then show how a generalized Kulkarni-Nomizu product encapsulates the operations involved both in our case and in the classical Gauss equation.


\subsection{Curvature in terms of the Weingarten operators}
\label{subsec:curvatureT}

\begin{lem}\label{lem:RinTT}
The Riemann curvature tensor of $(f(\cM),\III)$ and that of $(x(\cM),I)$ are related in terms of the covariant derivative difference vector $T$ by:
\begin{equation}\label{eq:RinTT}
R^{(\III)}(X,Y)Z - R(X,Y)Z = (\nabla_X T)(Y,Z) - (\nabla_Y T)(X,Z) + T(X,T(Y,Z)) - T(Y,T(X,Z)).
\end{equation}
\end{lem}
\begin{proof}
By definition of the curvature tensor in terms of second order covariant derivatives, and using $\nabla^{(\III)}_X=\nabla_X+T(X,\cdot)$, we have
\begin{align*}
R^{(\III)}(X,Y)Z & = \nabla^{(\III)}_X\nabla^{(\III)}_Y Z - \nabla^{(\III)}_Y\nabla^{(\III)}_X Z -\nabla^{(\III)}_{[X,Y]}Z \\
& =  \nabla^{(\III)}_X(\nabla_YZ + T(Y,Z)) - \nabla^{(\III)}_Y(\nabla_XZ + T(X,Z)) -\nabla_{[X,Y]}Z - T([X,Y], Z) \\
& = R(X,Y)Z + \nabla^{(\III)}_XT(Y,Z)- \nabla^{(\III)}_YT(X,Z)+T(X,\nabla_YZ)-T(Y,\nabla_XZ)-T([X,Y],Z).
\end{align*}
Now, recalling equation \ref{eq:brackets} and the linearity of $T$,
\begin{align*}
R^{(\III)}(X,Y)Z - R(X,Y)Z = & 
\nabla_X(T(Y,Z)) + T(X,T(Y,Z)) - \nabla_Y (T(X,Z)) - T(Y,T(X,Z)) \\
& \quad - T(Y,\nabla_XZ) + T(X,\nabla_YZ) -T(\nabla_XY,Z) + T(\nabla_YX,Z),
\end{align*}
which yields \ref{eq:RinTT} by collecting terms with $(\nabla_X T)(Y,Z)=\nabla_X(T(Y,Z))-T(\nabla_XY,Z)- T(Y,\nabla_XZ)$.
\end{proof}

The previous lemma realizes the curvature tensor of the Gauss image in terms of $T$, which from the results of the previous section is completely determined by the extrinsic shape operators. However, for hypersurfaces in non-flat space forms, the authors of \cite[Prop. 5.1]{Liu1999} obtain an expression resembling a \emph{theorema egregium} by writing it in terms of the second fundamental form. On the other hand \cite[eq. 5]{Aledo2005} shows that, for hypersurfaces with the second fundamental form as metric, the terms in $\nabla T$ have zero $\II$-trace of its Ricci contraction, whereas the terms of $T$ composed with itself can be decoupled by symmetries. We shall thus investigate \ref{eq:RinTT} in order to look for symmetries and further decomposition in arbitrary codimension.

{\bf Proof of Theorem \ref{th:main} (ii)}
We expand right hand side of equation \ref{eq:RinTT} when contracted with the $\III$ metric to yield $\III( R^{(\III)}(X,Y)Z - R(X,Y)Z,V)$.
First notice that the Codazzi symmetries of \ref{eq:Codazzi1} allow us to separate the composed $T$ operations:
\begin{align*}
\III( T(X,T(Y,Z)) -  T(Y,T(X,Z)), V ) & = \sum_j\langle A_j(\nabla_X A_j)T(Y,Z), V \rangle - \langle A_j(\nabla_Y A_j)T(X,Z), V \rangle \\
& = \sum_j\langle (\nabla_X A_j)(A_jV), T(Y,Z) \rangle - \langle (\nabla_Y A_j)(A_jV), T(X,Z) \rangle.
\end{align*}
By the covariant derivative of equation \ref{eq:Aderivative} and summing over $j$ these two terms transform into
\begin{align}
\nonumber \sum_j\langle\, (\nabla_X A_j)(A_jV), T(Y,Z) \,\rangle & = 
\sum_j\langle\, \nabla_X(A_j^2\, V)-A_j\nabla_X(A_j V), T(Y,X) \,\rangle
\\ \nonumber
& = \sum_j\langle\, \nabla_X(A^2_j\, V) - A_j^2\nabla_X V -A_j(\nabla_X A_j)V, T(Y,Z)  \,\rangle
\\
& =\langle\, (\nabla_X W)V, T(Y,Z)\,\rangle + \III( T(X,V), T(Y,Z) ) \label{eq:TT1}
\end{align}
and analogously
\begin{equation}\label{eq:TT2}
-\sum_j\langle\, (\nabla_Y A_j)(A_jV), T(X,Z) \,\rangle = -\langle\, (\nabla_Y W)V, T(X,Z)\,\rangle + \III( T(Y,V), T(X,Z) ).
\end{equation}
Recall that the summands dependent on the $\nabla T$ comprise the following terms
\begin{align}
\nonumber \III( \nabla_X(T(Y,Z)) -T(\nabla_XY,Z)-T(Y,\nabla_XZ) ,\, V ) - \III( \nabla_Y(T(X,Z)) -T(\nabla_Y X,Z)-T(X,\nabla_Y Z) ,\, V ).
\end{align}
These can be similarly written in $T$ and derivatives of $W$. Using the symmetry \ref{eq:ObataSym} and defining an auxiliary operator $N_X = \sum_j A_j(\nabla_X)A_j$, the $\nabla_X(T(X,Y))-\nabla_Y(T(X,Z))$ terms yield, using again the Codazzi equations:
\begin{align}
\III( \nabla_X(T(Y,Z) ,\, V ) & = \langle\, W\nabla_X(W^{-1}N_Y Z) ,\, V \,\rangle  \nonumber \\
& = \langle\, \nabla_X(WW^{-1}N_YZ) - (\nabla_XW)T(Y,Z),\, V\,\rangle \nonumber\\
& = \langle\, \nabla_X(N_YZ) ,\, V \,\rangle - \langle\, (\nabla_XW)V,\, T(Y,Z) \,\rangle, \label{eq:delT1}
\end{align}
and similarly
\begin{equation}\label{eq:delT2}
- \III(\nabla_Y(T(X,Z) ,\, V ) = - \langle\, \nabla_Y(N_XZ) ,\, V \,\rangle + \langle\, (\nabla_YW)V,\, T(X,Z) \,\rangle,
\end{equation}
so that the terms with derivatives $\nabla W$ cancel out when summing up \ref{eq:TT1}, \ref{eq:TT2}, \ref{eq:delT1} and \ref{eq:delT2}, resulting in the following expression for the right hand side of (\ref{eq:RinTT}) when $\III$-paired with $V$:
\begin{align}\label{eq:RHSmid}
\nonumber \III( T(\nabla_YX,Z) & + T(X,\nabla_YZ) -T(\nabla_XY,Z)-T(Y,\nabla_XZ) ,\, V )\; + \\
& + \langle\, \nabla_X(N_YZ)- \nabla_Y(N_XZ) ,\, V \,\rangle + \III( T(X,Z) ,\, T(Y,V)) - \III( T(X,V) ,\, T(Y,Z) ).
\end{align}
The first and third terms of the first line yield a Lie product term
$$
\III( T(\nabla_YX,Z) -T(\nabla_XY,Z),\, V ) = - \langle\, N_{[X,Y]}Z ,\, V \,\rangle,
$$
whereas the first and second terms of the second line of \ref{eq:RHSmid} expand as
$$
\langle\, (\nabla_X N_Y)Z + N_Y\nabla_X Z - (\nabla_YN_X)Z - N_X\nabla_YZ ,\, V \,\rangle.
$$
The second and fourth terms of \ref{eq:RHSmid} cancel out with the corresponding ones of the last equation, so \ref{eq:RHSmid} becomes
\begin{equation}\label{eq:RHSalmost}
\langle\, \left(\nabla_XN_Y +\nabla_YN_X-N_{[X,Y]} \right)Z ,\, V \,\rangle + \III( T(X,Z) ,\, T(Y,V)) - \III( T(X,V) ,\, T(Y,Z) ).
\end{equation}
Finally
\begin{align}
\nonumber & \nabla_XN_Y +\nabla_YN_X-N_{[X,Y]} = \sum_j \nabla_X(A_j\nabla_Y A_j) -\nabla_Y(A_j\nabla_X A_j) - A_j\nabla_{[X,Y]}A_j \\
\nonumber & = \sum_j (\nabla_X A_j)(\nabla_YA_j)+ A_j(\nabla_X\nabla_YA_j) -(\nabla_Y A_j)(\nabla_XA_j)- A_j(\nabla_Y\nabla_XA_j) - A_j(\nabla_{[X,Y]}A_j) \\
\nonumber & = \sum_j A_jR(X,Y)A_j + \left[\, \nabla_XA_j,\,\nabla_YA_j \,\right]
\end{align}
which proves equation \ref{eq:comparisonR} upon introducing $WW^{-1}$ to write it in terms of the metric $\III$.
\qed


\subsection{Generalized Kulkarni-Nomizu product}
\label{subsec:curvatureKN}

The usual Kulkarni-Nomizu product is a particularization of the product in the graded algebra $\bigoplus^n_{i=1}S^2(\Omega^i\cM)$ given on simple elements by
$$
(\alpha\cdot\beta)\KN(\gamma\cdot\delta) = (\alpha\wedge\gamma)\odot(\beta\wedge\delta),
$$
where $\odot$ is the symmetric product. For $h$ and $k$ symmetric $(0,2)$-tensors, and tangent vectors $X_1,X_2,X_3,X_4$, the Kulkarni-Nomizu product is:
\begin{equation}\label{eq:KNdef}
(h\KN k)(X_1,X_2,X_3,X_4) = \frac{1}{2}\left|\begin{array}{ll}
h(X_1,X_3) & h(X_1,X_4) \\
k(X_2,X_3) & k(X_2,X_4)
\end{array}\right|
+\frac{1}{2}
\left|\begin{array}{ll}
k(X_1,X_3) & k(X_1,X_4) \\
h(X_2,X_3) & h(X_2,X_4)
\end{array}\right|,
\end{equation}
sometimes the normalization by $1/2$ is not included. The Kulkarni-Nomizu product is of interest since it plays a role in the tensor decomposition of algebraic curvature tensors, see \cite{besse2007}.

The Gauss equation of a submanifold $\cM$ in some Riemannian manifold $\overline{\cM}$ relates the intrinsic Riemann curvature tensor of $\cM$ to that of the ambient manifold via a combination of products of the second fundamental form, a purely extrinsic object dependent on the embedding:
\begin{equation}\label{eq:Gauss}
\langle \overline{R}(X,Y)Z,V\rangle =  \langle R(X,Y)Z,V\rangle + \langle \II(X,Z),\II(Y,V)\rangle - \langle \II(X,V),\II(Y,Z)\rangle.
\end{equation}
In the classical case of surfaces in space, this equation results in the expression of the scalar curvature appearing in the \emph{theorema egregium}:
$$
K = \frac{LN-M^2}{eg-f^2},
$$
where $L, M, N$ are the independent components of the second fundamental form and $e, f, g$ those of the first fundamental form.
The combination of operations in \ref{eq:Gauss} appears as well in our main formula \ref{eq:comparisonR}, and resembles the expression of \ref{eq:KNdef} for $h=k$. If we extend the definition to bilinear forms valued in an inner product space $(E,g)$, so that the products of \ref{eq:KNdef} are understood as scalar $g$-products of the vector-valued forms $h,k$, then the operations carry over and generalize.
\begin{defi}\label{def:KN}
The generalized Kulkarni-Nomizu product $\KN_g$ of two vector-valued symmetric bilinear forms $h,k\in Sym^2(T_p\cM^*)\otimes E$, where $(E,g)$ is a finite-dimensional vector space with inner product $g$, is defined to be
\begin{align}\label{eq:KNdefGeneral}
(h\KN_g k)(X_1,X_2,X_3,X_4) = & \frac{1}{2}\left[ g(h(X_1,X_3),k(X_2,X_4)) - g(h(X_1,X_4),k(X_2,X_3)) \right.\\
& +\left. g(h(X_2,X_4),k(X_1,X_3)) - g(h(X_2,X_3),k(X_1,X_4))\nonumber
\right],
\end{align}
for any $X_1,X_2,X_3,X_4\in T_p\cM$.
\end{defi}
Using this expression, the usual Kulkarni-Nomizu product is $\KN_g$ when $(E,g)$ is the real field with the usual multiplication. Now, equation \ref{eq:Gauss} can be written as
$$
\langle \overline{R}(X,Y)Z,V\rangle = \langle R(X,Y)Z,V\rangle + (\II\KN_I \II) (X,Y,Z,V).
$$
where the inner product space in this case is the normal space to the manifold with the scalar product induced from the ambient Riemannian manifold, (which we have also called $I$ by abuse of notation to stress its relationship with the first fundamental form when restricted to the tangent space instead).

{\bf Proof of Corollary \ref{cor:main}}
In equation \ref{eq:comparisonR}, the terms
$$
\III ( T(X,Z),\, T(Y,V) ) -  \III( T(X,V),\, T(Y,Z))
$$
clearly come from $(T\KN_{\III}T)(X,Y,Z,V)$, since $T$ can be regarded as a bilinear form with values in the tangent space, with inner product given by the third fundamental form. Now, defining the symmetric bilinear forms $J_j$ valued in the tangent space as
$$J_j(X,Y)=W^{-\frac{1}{2}}(\nabla_X A_j)Y,$$
we can apply the Codazzi symmetires to the bracket terms from $P$ in \ref{eq:comparisonR} to obtain
\begin{align*}
& \III(W^{-1}\sum_{j=1}^m\left[ \nabla_X A_j,\,\nabla_Y A_j \right]Z,V) =\sum_{j=1}^m\langle\, (\nabla_X A_j)V,(\nabla_Y A_j)Z \,\rangle -\langle\, (\nabla_Y A_j)V,(\nabla_X A_j)Z \,\rangle \\
& = \sum_{j=1}^m \III( J_j(X,V),\, J_j(Y,Z))- \III ( J_j(X,Z),\, J_j(Y,V) ) = - \sum_{j=1}^m (J_j\KN_{\III} J_j)(X,Y,Z,V).
\end{align*}
\qed


\subsection{Scalar curvature of the Gauss image}
\label{subsec:curvatureScalar}


{\bf Proof of Theorem \ref{th:main} (iii)}
To compute the scalar curvature of the Gauss map we just need to contract the tensors in equation \ref{eq:comparisonR} with respect to a $\III$-orthonormal basis $\{V_\mu\}_{\mu=1}^n$ of $T_p\cM$ at every $p\in\cM$. Since for normally flat submanifolds the Weingarten operators can be simultaneously diagonalized by a tangent $I$-orthonormal basis $\{U_\mu\}_{\mu=1}^n$, we have real numbers $\lambda_{j\mu}$ such that
\begin{equation}
A_{\xi_j}U_\mu = \lambda_{j\mu} U_\mu\text{ such that } \langle U_\mu, U_\nu\rangle = \delta_{\mu\nu},\quad j =1,\dots,m,\mu =1,\dots,n.
\end{equation}
We choose $V_\mu=w_\mu^{-1/2}U_\mu$ with $w_\mu = \sum_{j=1}^m\lambda_{j\mu}^2$, for all $\mu$, so that 
\begin{equation}
W V_\mu = w_\mu V_\mu,\text{ and } \III( V_\mu, V_\nu)= \langle w_\mu\frac{U_\mu}{\sqrt{w_\mu}}, \frac{U_\nu}{\sqrt{w_\nu}}\rangle= \delta_{\mu\nu},\quad j =1,\dots,m,
\end{equation}
and thus $\{V_\mu\}_{\mu=1}^n$ is the $\III$-orthonormal basis to work with.

Now, by definition of Ricci contraction $\tr_c R=\tr(V\mapsto R(V,X)Y)$ with respect to the corresponding metric, the self-adjointness of $W$, the definition of $V_\mu$ and the linearity of $R$, we have
\begin{align*}
\mathcal{R}ic^{(\III)}(X,Y) = \sum_{\mu=1}^n \III(R^{(\III)}(V_\mu,X)Y,V_\mu) =  \sum_{\mu=1}^n \langle R^{(\III)}(V_\mu,X)Y, W V_\mu\rangle =  \sum_{\mu=1}^n \langle R^{(\III)}(U_\mu,X)Y, U_\mu\rangle.
\end{align*}
Applying our fundamental curvature relation of equation \ref{eq:comparisonR} leads to
\begin{align*}
\mathcal{R}ic^{(\III)}(X,Y) = \mathcal{R}ic(X,Y) + \sum_{\mu=1}^n\left(\; \langle P(U_\mu,X)Y,U_\mu\rangle +\III(T(V_\mu,Y),T(V_\mu,X)) -\III(T(V_\mu,V_\mu),T(X,Y)) \;\right)
\end{align*}
We obtain the scalar curvature by further contracting this equation :
\begin{align*}
\mathcal{R}^{(\III)} = \sum_{\nu=1}^n\mathcal{R}ic^{(\III)}(V_\nu,V_\nu) = \sum_{\nu=1}^n \mathcal{R}ic(V_\nu,V_\nu) & + \sum_{\nu=1}^n\sum_{\mu=1}^n\left(\; \langle P(U_\mu,V_\nu)V_\nu,U_\mu\rangle\right. \\ 
 & \left. +\III(T(V_\mu,V_\nu),T(V_\mu,V_\nu)) -\III(T(V_\mu,V_\mu),T(V_\nu,V_\nu)) \;\right),
\end{align*}
where the second summand is the $\III$-trace of the Ricci contration of $P$, i.e., $\tr_{\III}\tr_c P$, the third summand is the $\III$-norm squared of the tensor $T$ (the analogous to the Frobenius norm squared of a matrix), and the last summand is the $\III$-norm squared of the vector $\tr_{\III}T = \sum_{\mu=1}^n T(V_\mu,V_\mu)$. That is, we obtain
\begin{align*}
& \mathcal{R}^{(\III)} =   \sum_{\nu=1}^n \mathcal{R}ic(V_\nu,V_\nu) + \tr_{\III}\tr_c P + ||T||^2_{\III} - ||\tr_{\III}T||_{\III}^2.
\end{align*}
Finally, using the relationship of equation \ref{eq:Obata}, we identify 
\begin{align*}
\sum_{\nu=1}^n\mathcal{R}ic(V_\nu,V_\nu) & = k(n-1)\sum_{\nu=1}^n\langle V_\nu,V_\nu\rangle + \langle\; \sum_{\nu=1}^n\II(V_\nu,V_\nu), H \rangle - \sum_{\nu=1}^n\III(V_\nu,V_\nu) \\
& = k(n-1)\sum_{\nu=1}^n\frac{1}{w_\nu} + \sum_{\nu=1}^n\langle\, A_H V_\nu,V_\nu\rangle - n.
\end{align*}
Here, the second term is clearly $\tr_{\III}A_H$, and the first term satisfies
$$
\sum_{\nu=1}^n\frac{1}{w_\nu} = \frac{\sum_{1\leq i_1<\dots<i_{n-1}\leq n}^n w_{i_1}\cdots w_{i_{n-1}}}{\prod_{\nu=1}^n w_\nu} = \frac{\sigma_{n-1}(W)}{\sigma_n(W)},
$$
writing $\sigma_{n-1}(W),\sigma_n(W)=\det(W)$ for the elementary symmetric polynomials on the eigenvalues of W, and we are done.
\qed


%

%
%
%
%
%


\section{Concluding remarks}

In the present work we have established formulas for the intrinsic geometry of the Gauss map via the third fundamental form as generalized by Obata, for a submanifold with flat normal bundle in a space form satisfying general regularity conditions. We have found a relationship analogous to a \emph{theorema egregium} that yields the difference of the intrinsic Riemann-Christoffel curvature tensors of the Gauss image and the original submanifold in terms of tensors solely constructed from the extrinsic Weingarten operators of the original submanifold. For this, the explicit dependency of the difference of the induced covariant derivatives on the shape operators was obtained, along with the introduction of a generalized Kulkarni-Nomizu product. The scalar curvature of the Gauss map in terms of extrinsic scalars was established as a consequence. 
Once the geometry of the Gauss map has been established, applications to the study of normally flat submanifolds may be possible, along with the investigation of extending the results to other types of submanifolds. For example, a natural line of research would be the study of the relationship between the mean curvature of the Gauss image, as defined by normal deformations of the Gauss map, and the third fundamental form scalar curvature, in order to establish characterizations similar to the case of the classical ovaloids and extrinsic spheres. 




\bibliographystyle{plainnat}
\bibliography{journal-article}

\end{document}